\documentclass[a4paper, 12pt]{article}
\usepackage{amsmath,amsthm,amssymb}

\input xy
\xyoption{all}

\usepackage{color}
\usepackage{makeidx}
\usepackage{hyperref}
\makeindex 

\numberwithin{equation}{section}

{\theoremstyle{definition}\newtheorem{definition}{Definition}[section]

\newtheorem{remark}[definition]{Remark}

}
\newtheorem{proposition}[definition]{Proposition}
\newtheorem{proposition-definition}[definition]{Proposition-Definition}

\newtheorem{theorem}[definition]{Theorem}

\newcommand{\R}{\mathbb{R}}

\newcommand{\cF}{\mathcal{F}}
\newcommand{\cG}{\mathcal{G}}

\newcommand{\C}{\mathbb{C}}

\newcommand{\ev}{{\rm ev}}

\newcommand{\cA}{\mathcal{A}}
\newcommand{\cC}{\mathcal{C}}

\newcommand{\cP}{\mathcal{P}}
\newcommand{\cU}{\mathcal{U}}
\newcommand{\cH}{\mathcal{H}}

\newcommand{\cT}{\mathcal{T}}

\newcommand{\id}{{\hbox{id}}}
\newcommand{\ind}{{{\mathrm{ind}}}}
\newcommand{\ie}{{\it i.e.}\/ }
\newcommand{\eg}{{\it e.g.}\/ }
\newcommand{\cf}{{\it cf.}\/ }

\def\gpd{\,\lower1pt\hbox{$\longrightarrow$}\hskip-.24in\raise2pt
             \hbox{$\longrightarrow$}\,}

\begin{document}

\renewcommand\theenumi{\alph{enumi}}
\renewcommand\labelenumi{\rm {\theenumi})}
\renewcommand\theenumii{\roman{enumii}}

\begin{center}
{\Large\bf The analytic index of elliptic pseudodifferential operators on a singular foliation
\footnote{AMS subject classification: Primary 57R30, 46L87.
Secondary 46L65.}
\footnote{Research partially supported by DFG-Az:Me 3248/1-1.}


\bigskip

{\sc by Iakovos Androulidakis and Georges Skandalis}
}
\end{center}

{\footnotesize
 Georg-August Universit\"{a}t G\"{o}ttingen
\vskip -4pt Mathematisches Institut
\vskip -4pt Bunsenstrasse 3-5
\vskip -4pt D-37073 G\"{o}ttingen, Germany
\vskip -4pt e-mail: iakovos@uni-math.gwdg.de

\vskip 2pt Institut de Math{\'e}matiques de Jussieu, UMR 7586
\vskip -4ptCNRS - Universit\'e Diderot - Paris 7
\vskip-4pt 175, rue du Chevaleret, F--75013 Paris
\vskip-4pt e-mail: skandal@math.jussieu.fr
}
\bigskip
\everymath={\displaystyle}

\textbf{This draft: \today}

\begin{abstract}\noindent
In previous papers (\cite{AndrSk, AndrSkPsdo}), we defined the $C^*$-algebra and the longitudinal pseudodifferential calculus of any singular foliation $(M,\cF)$. Here we construct the analytic index of an elliptic operator as a $KK$-theory element, and prove that the same element can be obtained from an ``adiabatic foliation'' $\cT\cF$ on $M \times \R$, which we introduce here.
\end{abstract}

\section*{Introduction}

This article follows \cite{AndrSk} and \cite{AndrSkPsdo} regarding the study of singular foliations $(M,\cF)$. In these papers we constructed
\begin{itemize}
\item the holonomy groupoid $\cG(M,\cF)$, which is a topological groupoid endowed with a usually ill-behaved (quotient) topology;
\item the (full and reduced) $C^{*}$-algebra of the foliation $(M,\cF)$;
\item  the extension of $0$-order pseudodifferential operators: this is a short exact sequence $$0 \to C^{*}(M,\cF) \to {\Psi (M,\cF)} \stackrel{\sigma}{\longrightarrow} B \to 0 \eqno (\ref{zeroextn})$$ where $B$ is a quotient of the algebra of continuous functions on a cosphere ``bundle'' naturally associated with $\cF$.
\end{itemize}
The key to these constructions is the notion of a bi-submersion, which we are going to use here as well. In a broad sense this may be thought of as a cover of an open subset of the holonomy groupoid; it is given by a manifold $U$ and two submersions $s,t : U \to M$, each of which lifts the leaves of $\cF$ to the fibers of $s$ and $t$. 

\bigskip Here we study the analytic index of elliptic longitudinal pseudodifferential operators. This is a map $K_0(C_0(\cF^*))\to K_0(C^*(M,\cF))$ which can be directly expressed in terms of the extension of $0$-order pseudodifferential operators: indeed, the $K$-theory of $C_0(\cF^*)$ can be identified with the relative $K$-theory of the morphism $p:C(M)\to C(S\cF^*)$, and a natural commuting diagram gives rise to a map $K_0(\cF^*)=K_0(p)\to K_0(\sigma)=K_0(C^*(M,\cF))$ which is the analytic index. Moreover, using mapping cones, we construct this morphism as an element $\ind_a\in KK(C_0(\cF^*),C^{*}(M,\cF))$.

We then  prove that it can be obtained from a tangent groupoid, in the spirit of \cite{Connes} and \cite{Monthubert-Pierrot}:
\begin{itemize}
\item Every foliation $(M,\cF)$ gives rise to an ``adiabatic'' foliation $\cT\cF$ on $M \times \R$.
\item The holonomy groupoid of $\cT\cF$ is a ``deformation groupoid'', namely $$\cG_{\cT\cF} =  (\bigcup_{x \in M}\cF_{x}) \times \{0\} \cup \cG_{\cF} \times \R^{*}.$$
\item Restricting $C^*(M\times \R,\cT\cF)$ to the interval $[0,1]$, we find an extension $$0 \to C_{0}((0,1]) \otimes C^{*}(M,\cF) \to C^{*}(M \times [0,1],\cT\cF) \stackrel{\ev_{0}}{\longrightarrow} C_{0}(\cF^{*}) \to 0 \eqno (\ref{extn2})$$
\item The morphism $\ev_0$ has a contractible kernel; it is a $KK$-equivalence. We finally establish the equality $$\ind_a = [\ev_{0}]^{-1}\otimes [\ev_{1}].$$
\end{itemize}

\bigskip The material is structured as follows:
\begin{enumerate}
\item To make the paper self-contained, we start in section 1 with a brief overview of various definitions and results in \cite{AndrSk} and \cite{AndrSkPsdo}.

\item In section 2 we show how extension (\ref{zeroextn}) gives rise to the analytic index. This is obtained as the class  of the morphism of mapping cones $C_0(\cF^*)=\cC_p\to \cC_\sigma$ under the $KK$-equivalence associated with the ``excision map'' $e:C^*(M,\cF)=\ker \sigma\to \cC_\sigma$. Finally, we briefly explain \begin{enumerate}
\item how this construction relates to the element in $KK^1(C(S\cF^*),C^*(M,\cF))$ associated with extension (\ref{zeroextn});
\item  we show that the $K$-theory map $K_0(C_0(\cF^*))\to K_0(C^*(M,\cF))$ associated with $\ind_a$ is indeed the (analytic) index map - using relative $K$ theory.
\end{enumerate}

\item In section 3 we introduce the ``adiabatic foliation'' $(M \times \R,\cT\cF)$, whose leaves are $L\times \{\beta\}$ where $L$ is a leaf of $(M,\cF)$ and $\beta\in \R^*$ as well as points $\{(x,0)\}$ ($x\in M$) (\footnote{Recall however that in the case of singular foliations, the partition into leaves doesn't determine the foliation. The precise definition of $\cT\cF$ is given in section 3.}). We show that its bi-submersions are deformations to the normal cone of identity sections in bi-submersions of $(M,\cF)$. From this, it follows that the holonomy groupoid of $\cT\cF$ is the deformation groupoid of $\cG_{\cF}$ we mentioned before. Finally, we show that $C^{*}(M \times \R,\cT\cF)$ lies in a natural exact sequence as well as the extension discussed above.

\item In section 4, we examine the extension of $0$-order pseudodifferential operators on the adiabatic foliation. We deduce that the analytic index can be obtained from the adiabatic foliation.
\end{enumerate}

\paragraph{Conventions and notation}\label{section1.6}

\begin{enumerate}
\item In our notation we assume that the manifold $M$ is compact: we write $C(M)$, $C(S\cF^*)$, \textit{etc.}\/ and talk about the \emph{compactification} $\overline{\cF^*}=\cF^*\cup S\cF^*$ of $\cF$. On the other hand, everything remains true in the non compact case: one has just to work with algebras of continuous functions that vanish at infinity, and consider the closure of the set of pseudodifferential operators with compact support. This is already used in the case of the \emph{adiabatic foliation} whose base space is $M\times \R$ (although we right away restrict to $M\times [0,1]$).

\item \label{section1.6.b}We fix an atlas $\cU$ for our foliation $(M,\cF)$, which could as well be the smallest one - the \emph{path homotopy atlas} and write $C^*(M,\cF)$ instead of $C^*(\cU)$. Actually, this is not an important issue, since we have a natural morphism $C^*(M,\cF)\to C^*(\cU)$ for any atlas, and the index for $C^*(\cU)$ is just the push-forward by this morphism of  the index for $C^*(M,\cF)$.
\end{enumerate}

\tableofcontents

\section{Singular foliations and $C^*$-algebras}

We briefly recall here some facts and constructions from  \cite{AndrSk, AndrSkPsdo}.

\subsection{Foliations}

\begin{definition}\begin{enumerate}
\renewcommand{\theenumi}{\alph{enumi}}
\renewcommand{\labelenumi}{\theenumi)}

\item Let $M$ be a smooth manifold. A \emph{foliation} on $M$ is a
locally finitely generated submodule of $C_c^\infty(M;TM)$ stable
under Lie brackets.

\item For $x\in M$, put $I_{x} = \{ f \in C^{\infty}(M) : f(x)=0 \}$. The \textit{fiber} of $\cF$ is the quotient $\cF_{x} = \cF/I_{x}\cF$. The \textit{tangent space of the leaf} is the image $F_{x}$ of the evaluation map $\ev_{x} : \cF \to T_{x}M$.

\item The {\em cotangent ``bundle''} of the foliation $\cF$ is the union  $\cF^{*}=\coprod_{x\in M} \cF^*_x$. It has a natural projection $p:\cF^*\to M$ ($(x,\xi)\mapsto x$) and for each $X\in \cF$, there is a natural map $q_X:(x,\xi)\mapsto \xi\circ e_x(X)$. We endow $\cF^*$ with the weakest topology for which the maps $p$ and $q_X$ are continuous. This makes it a locally compact space (\cf \cite[\S 2.2]{AndrSkPsdo})

\end{enumerate}
\end{definition}

If
$(M,\cF)$ is a foliation and $f:M\times L\to M$ is the first
projection, every vector field $X$ of $M$ extends to a vector field $X \otimes 1$ on $M \times L$, which is tangent along $M$. We define the foliation $\cF\otimes 1$ to be the submodule of $C_c^\infty (M\times L;T(M\times L))$  which consists of finite sums $\sum f_i(X_i\otimes 1)$ where $f_i\in C_c^\infty (M\times L)$ and $X_i\in \cF$.
The pull back foliation $f^{-1}(\cF)$ is the space of
vector fields spanned by vertical vector fields ($\ker df$) and $\cF\otimes 1$. In this way, we define also the pull-back  foliation  by a submersion.

\subsection{Bi-submersions}

The key ingredient in our study of the holonomy of a foliation is the notion of a bi-submersion. This can be thought of as a piece of the holonomy groupoid. Explicitly:

\begin{definition}
A \emph{bi-submersion} of $(M,\cF)$ is a
smooth manifold $U$ endowed with two smooth maps $s,t:U\to M$  which are submersions and
satisfy:\begin{enumerate}\renewcommand{\theenumii}{\roman{enumii}}
\renewcommand{\labelenumii}{(\theenumii)}
\item $s^{-1}(\cF)=t^{-1}(\cF)$.
\item $s^{-1}(\cF)=C_c^\infty(U;\ker ds)+C_c^\infty (U;\ker dt)$.
\end{enumerate}
If  $(U,t,s)$ is a bi-submersion then the dimension of the manifold $U$ is at at least $\dim M + \dim \cF_{s(u)}$. We say it is \em{minimal} at $u \in U$ if $\dim U = \dim M + \dim \cF_{s(u)}$.
\end{definition}
If $(U,t_{U},s_{U})$ and $(V,t_{V},s_{V})$ are bi-submersions then $(U,s_{U},t_{U})$ is a bi-submersion - called the \emph{inverse bi-submersion} and noted $U^{-1}$, as well as $(W,s_{W},t_{W})$ where $W = U \times_{s_{U},t_{V}}V$, $s_{W}(u,v) = s_{V}(v)$ and $t_{W}(u,v) = t_{U}(u)$  - called the \emph{composition} of $U$ and $V$ and denoted $U\circ V$ (\cite[Prop. 2.4]{AndrSk}). 

\begin{definition}[morphisms of bi-submersions]  Let $(U_i,t_i,s_i)$ ($i=1,2$) be bi-submersions. A smooth map $f:U_1\to U_2$ is a \emph{morphism of bi-submersions} if $s_1=s_2\circ f$ and $t_1=t_2\circ f$.
\end{definition}

A notion which is very important for the pseudodifferential calculus is that of an identity bisection:

\begin{definition}  An \textit{identity bisection} of $(U,t,s)$ is a locally closed submanifold $V$ of $U$ such that the restrictions to $V$ of $s$ and $t$ coincide and are \'etale. 
\end{definition}

Also, for every bi-submersion $(U,t,s)$ and every $u\in U$, there exist  a neighborhood $W$ of $u$ in $U$, a bi-submersion $(U',t',s')$ and a submersion which is a morphism $f:(W,t_{|W},s_{|W})\to (U',t',s')$ such that $U'$ is minimal at $f(u)$.

\subsection{The groupoid of an atlas}

\begin{definition}
 Let $\cU = \big((U_{i},t_{i},s_{i})\big)_{i \in I}$ be a family of bi-submersions. A bi-submersion $(U,t,s)$ is \emph{adapted} to $\cU$ if for all $u \in U$ there exists an open subset $U' \subset U$ containing $u$, an $i \in I$, and a morphism of bi-submersions $U' \to U_{i}$. 

We say that $\cU$ is an \textit{atlas} if
\begin{enumerate}
\item $\bigcup_{i \in I}s_i(U_{i}) = M$.
\item The inverse of every element in $\cU$ is adapted to $\cU$.
\item The composition $U \circ V$ of any two elements in $\cU$ is adapted to $\cU$.
\end{enumerate}
An atlas $\cU' = \{(U'_{j},t'_{j},s'_{j})\}_{j \in J}$ is adapted to $\cU$ if every element of $\cU'$ is adapted to $\cU$. We say $\cU$ and $\cU'$ are \textit{equivalent} if they are adapted to each other. There is a \textit{minimal atlas} which is adapted to any other atlas: this is the atlas generated by ``identity bi-submersions''. 
\end{definition}

The groupoid $\cG(\cU)$ of an atlas $\cU = \big((U_{i},t_{i},s_{i})\big)_{i \in I}$ is the quotient of $U=\coprod_{i\in I}U_i$ by the equivalence relation for which $u\in U_i$ is equivalent to $u'\in U_j$ if there is a morphism of bi-submersions $f:W\to U_j$ defined in a neighborhood $W\subset U_i$ of $u$ such that $f(u)=u'$.

\subsection{The $C^*$-algebra of a foliation}

In \cite[\S 4]{AndrSk} we associated to an atlas $\cU$ its (full) $C^*$-algebra $C^*(\cU)$. To any bi-submersion $W$ adapted to $\cU$ we associate a map $Q_{W} : C^\infty_{c}(W;\Omega^{1/2}W) \to C^*(\cU)$, where  $\Omega^{1/2}W$ is the bundle of half densities on $\ker ds\times \ker dt$. The image $\bigoplus_{i \in I}Q_{U_i}(C^\infty_{c}(U_{i};\Omega^{1/2}U_{i}))$ is a dense $*$-subalgebra of $C^*(\cU)$.

When $\cU$ is the minimal atlas this algebra is denoted by $C^{*}(M,\cF)$.

In \cite[\S5]{AndrSk} it was shown that the representations of the full $C^{*}$-algebra of an atlas on a Hilbert space correspond to representations of the associated groupoid on a Hilbert bundle (desintegration theorem). We are going to use this correspondence in this sequel, so let us recall the explicit definition of these groupoid representations:

\begin{definition}
Let $\cU = \{(U_i,t_i,s_i)\}_{i \in I}$ an atlas of the foliation $(M,\cF)$. A representation of $\cG(\cU)$ is a triple $(\mu,H,\chi)$ where:
\begin{enumerate}
 \item $\mu$ is a quasi-invariant measure on $M$. Namely, for every $(U,t,s) \in \cU$ and positive Borel sections $\lambda^s$ of $\Omega^1(\ker ds)$ and $\lambda^t$ of $\Omega^1(\ker dt)$, the measures $\mu \circ \lambda^s$ and $\mu \circ \lambda^t$ are equivalent.
 \item $H = (H_x)_{x \in M}$ is a $\mu$-measurable field of Hilbert spaces over $M$.
 \item $\chi = \{\chi^U\}_{\cU}$ is a family of $\mu\circ\lambda$-measurable sections of unitaries $\chi^U_u : H_{s(u)} \to H_{t(u)}$. Moreover $\chi$ is a homomorphism defined in $\cG(\cU)$. That is to say: 
\begin{itemize}
\item if $f : U \to U'$ is a morphism of bi-submersions then $\chi_{f(u)}^{U'} = \chi_u^U$ for almost all $u \in U$, and 
\item $\chi^{U\circ V}_{(u,v)} = \chi^U_u \chi^V_v$ for almost all $(u,v) \in U \circ V$, for all bi-submersions $U,V$  adapted to $\cU$.
\end{itemize}
\end{enumerate}

\end{definition}


\subsection{The extension of pseudodifferential operators of order $0$}

In \cite{AndrSkPsdo}, we constructed the pseudodifferential calculus in the context of singular foliations. In fact, we will use here very little information on these operators,  just the exact sequence of order $0$ pseudodifferential calculus. 

This is an exact sequence of $C^{*}$-algebras
\begin{align}\label{zeroextn}
0 \to C^{*}(M,\cF) \to {\Psi(M,\cF)} \to B(M,\cF) \to 0
\end{align}
where ${\Psi(M,\cF)}$ is the $C^*$-algebra of the zero-order pseudodifferential operators and $B$ is the commutative $C^*$-algebra of symbols of order $0$. It is a quotient of $C(S\cF^{*})$ (the continuous functions on the cosphere ``bundle'').

\section{The analytic index}

The analytic index of elliptic pseudodifferential operators on a Lie groupoid $G$ over $M$ (\cf \cite{Monthubert-Pierrot}, \cite{Nistor-Weinstein-Xu}) is a group morphism defined on $K_{0}(C_{0}(A^{*}G))$, with values in $K_{0}(C^{*}(G))$.  It maps the class of the principal symbol of an elliptic pseudodifferential operator $P$, which is an element of $K_0(C_{0}(A^{*}G))$, to the index class of $P$, \ie an element of $K_{0}(C^{*}(G))$. 

The exact sequence (\ref{zeroextn}) of the zero-order  pseudodifferential calculus, allows us to extend this construction to the framework of singular foliations. This map comes from an element $\ind_a\in KK(C_0(\cF^*),C^*(M,\cF))$. 
Since we wish to identify precisely this element with the one obtained using the ``tangent groupoid'', we will give this construction explicitly, based on mapping cones and the identifications involved.

Let us also point out that the analytic index for smooth groupoids is sometimes presented as the connecting map associated with the exact sequence  (\ref{zeroextn}) (or the $KK^1$ element this exact sequence defines). Our presentation has two minor advantages:
\begin{itemize}
\item It is slightly more primitive since the element in $KK^1(C_0(S\cF^*),C^*(M,\cF))$ is in fact the composition of our $\ind_a$ with the $KK^1$ element corresponding with the obvious extension of $C_0(S\cF^*)$ by $C_0(\cF^*)$ (the one defined by the compactification $\cF^* \cup S\cF^*$);

\item Our element is slightly more tractable and has no sign problems since it only involves homomorphisms of $C^*$-algebras.
\end{itemize}

\subsection{Mapping cones}

We briefly recall here some facts about mapping cones and their use in $KK$-theory (\cf \cite{Kasc, Cuntz-Skandalis}). Let $\varphi : A \to B$ a homomorphism of unital $C^{*}$-algebras. 

\begin{itemize}
\item The mapping cone of $\varphi$ is the $C^{*}$-algebra $\cC_{\varphi} = \{(f,a) \in C_{0}([0,1);B) \times A;\ \varphi(a) = f(0)\}$.

\item The cone of the $C^{*}$-algebra $B$ is the contractible $C^{*}$-algebra $\cC_{B} =\cC_{\id_B}= C_{0}([0,1);B)$. If $\varphi : A \to B$ is onto, we have the exact sequence $0 \to \ker\varphi \ {\buildrel e\over {\longrightarrow}}\ \cC_{\varphi} \to \cC_{B} \to 0$, where $e(x)=(0,x)$ for $x\in \ker \varphi$. The 6-term exact sequence gives $K_{0}(\cC_\varphi) = K_{0}(\ker\varphi)$. If  $\varphi$ admits a completely positive splitting, the element $[e]\in KK(\ker \varphi,\cC_\varphi)$ is invertible.

\item The ``cone'' construction is natural, namely a commutative diagram of $C^{*}$-algebra homomorphisms
$$
\xymatrix{
A \ar[d] \ar[r]^{\varphi} & B \ar[d] \\ 
A' \ar[r]^{\varphi'} & B'
}
$$
gives rise to a $*$-homomorphism $\cC_{\varphi} \to \cC_{\varphi'}$.

\item Let $\varphi : Y \to Z$ be a proper map between locally compact spaces. The mapping cone of $\varphi$ is $$C_{\varphi} = Y \times [0,1)\cup Z /_\sim$$ where the equivalence relation is  $(y,0) \sim \varphi(y)$ for all $y \in Y$. Abusing the notation, we write $\varphi : C_0(Z) \to C_0(Y)$ for the induced map. The cone of this $\varphi$ is the algebra of continuous functions on the mapping cone, \ie $ \cC_{\varphi}=C_0(C_{\varphi}) $.
\end{itemize}

\subsection{Analytic index}

Let us now come to the case of a (singular) foliation. Locally, $\cF^*$ is a closed subspace of (the total space of) a vector bundle. We may choose a metric on this bundle; this will fix continuously a euclidian metric on each $\cF^*_x$. This can even be done globally using partitions of the identity. Let then $S\cF^{*}$ be the sphere ``bundle'' of $\cF^{*}$, \ie the space of half lines in $\cF^*$ identified with the space of vectors of norm $1$ in $\cF^*$.

We obviously have:

\begin{proposition}\label{cone}
Let $(M,\cF)$ be a foliation. The cone of the projection $p : S\cF^{*} \to M$ is canonically isomorphic with $\cF^*$. \hfill$\square$
\end{proposition}

Every function $f \in C(M)$ can be considered as the zero-order longitudinal (pseudo)differential operator $m(f)$ which acts by multiplication on the algebra $\cA_{\cU}$. Its principal symbol is constant on covectors $\sigma(m(f))(x,\xi)=f(x)$ for $x\in M$ and $\xi \in \cF^*_x$. In other words, we have a commutative diagram
\begin{align}\label{diag2}
\xymatrix{
\ & C(M) \ar[r]^{p} \ar[d]_{m} & C(S\cF^{*}) \ar[d]^{q} & & \  \\ 
0 \to C^{*}(M,\cF) \ar[r] & \Psi(M,\cF)  \ar[r]^{\sigma} & B(M,\cF) \ar[r] & 0
}
\end{align}

Denote by $\varphi : C_0(\cF^*)\simeq \cC_p\to \cC_\sigma$ the homomorphism  induced by the commuting square
$$
\xymatrix{
C(M) \ar[r]^{p} \ar[d]_{m} & C(S\cF^{*}) \ar[d]^{q} \\ 
{\Psi(M,\cF)}  \ar[r]^{\sigma} & B(M,\cF)
}
$$

Denote by $e:C^{*}(M,\cF)\to \cC_\sigma$ the ``excision'' map associated with the exact sequence (\ref{zeroextn}). It is a $KK$-equivalence since $B(M,\cF)$ is abelian, whence  $\sigma $ admits a completely positive cross-section.

\begin{definition}
The analytic index is the element $$\ind_a=[\varphi]\otimes_{\cC_\sigma}[e]^{-1}\in KK(C_0(\cF^*),C^*(M,\cF)).$$
\end{definition}

\subsection{The element of $KK^1$ associated with the extension  (\ref{zeroextn})}

Let $0\to J\to A\ {\buildrel p\over \longrightarrow}\  A/J\to 0$ be an exact sequence of $C^*$-algebras. Assume that the morphism $p$ admits a completely positive section. Consider the morphisms $e:J\to \cC_p$ and $j:C_0((0,1);A/J)\to \cC_p$ given by $e(x)=(0,x)$ and $j(f)=(f,0)$. Recall that the element of $KK_1(A/J,J)=KK(C_0((0,1);A/J),J)$ associated with the exact sequence above is the composition $[j]\otimes_{\cC_p}[e]^{-1}$.

It follows that the element of $KK^1(C_0(S\cF^*),C^*(M,\cF))$ associated with the exact sequence (\ref{zeroextn}) is just $j^*(\ind_a)$ where $j:C_0(S\cF^*\times (0,1))\simeq C_0(\cF^*\setminus M)\to C_0(\cF^*)$ is the inclusion.

In this way, the element $\ind_a$ we just constructed is more primitive than the $KK^1$ element associated with extension  (\ref{zeroextn}).

\subsection{The $K$-theory map associated with $\ind_a$}

We now show that the $K$-theory map $K_0(C_0(\cF^*))\to K_0(C^*(M,\cF))$ associated with $\ind_a$ is indeed the index. 

Another way of looking at the $K$-theory of a cone is relative $K$-theory. 

Let $\varphi : A \to B$ be a homomorphism of unital $C^{*}$-algebras. 

Recall that the group $K_{0}(\varphi)$ is given by generators and relations: \begin{itemize}
\item Its generators are triples $(e^{+},e^{-},u)$ where $e^{+}, e^{-} \in M_{n}(A)$ are idempotents and $u \in M_{n}(B)$ is such that $uv = \varphi(e^{+})$ and $vu = \varphi(e^{-})$ for some $v \in M_{n}(B)$.  
\item Addition is given by direct sums.

\item Trivial elements are  those triples $(e^{+},e^{-},u)$ for which $u = \varphi(u_{0})$ and $v = \varphi(v_{0})$ for some $u_{0}, v_{0} \in M_{n}(A)$ satisfying $u_{0}v_{0} =e^+$ and $v_{0}u_{0}=e^-$.

\item The group $K_0(\varphi)$ is formed as the set of those triples divided by trivial triples, and homotopy - which is given by triples associated with the map $C([0,1];A)\to C([0,1];B)$ associated with $\varphi$.

\item For non unital algebras / morphisms, we just put $K_{0}(\varphi)=K_{0}(\tilde \varphi)$, where $\tilde \varphi:\tilde A\to \tilde B$ is obtained by adjoining units everywhere.
\end{itemize}

Note that $K_0(A)=K_0(A\to 0)=K_0(\epsilon _A)$ where $\epsilon _A:\tilde A\to \C$ is the morphism with kernel $A$.

If $\varphi$ is onto, then $K_0(\varphi)=K_0(\ker \varphi)$. More precisely, the map $K_0(\epsilon_{\ker \varphi})\to K_0(\varphi)$ induced by the commuting diagram$$
\xymatrix{
\widetilde{\ker \varphi} \ar[d] \ar[r]^{\epsilon_{\ker \varphi}} & \C \ar[d] \\ 
\tilde A \ar[r]^{\varphi} & \tilde B
}
$$
is an isomorphism. The inverse of this isomorphism is the \emph{index map} of the exact sequence.

Finally, there is a natural isomorphism $K_0(\varphi)\to K_0(\cC_\varphi)$ which in case $\varphi$ is onto, is compatible with the identifications of $K_0(\varphi)$ and of $K_0(\cC_\varphi)$ with $K_0(\ker \varphi)$

A symbol of an elliptic operator of order $0$ is given by two bundles $E_\pm$ and an isomorphism $a$ of the pull-back bundles $p^*(E_\pm)$, where $p:S^*\cF\to M$ is the projection. It therefore defines an element of $K_0(p)\simeq K_0(C_0(\cF^*))$. A diagram chasing shows that its image by $\ind_a$ is indeed the index class of a pseudodifferential operator with symbol $a$.

\section{The tangent groupoid}


We now construct the \emph{adiabatic foliation} associated to a foliation $\cF$. This is going to be a foliation on $M\times \R$. The holonomy groupoid of this foliation is $\bigcup_{x\in M}\cF_x\times \{0\}\cup \cG\times \R^*$. It is called the \emph{tangent groupoid.} Its $C^*$-algebra contains as an ideal $C_0(\R^*)\otimes
C^*(M;\cF)$ with quotient $C_0(\cF^*)$.

This tangent groupoid allows to construct an element in $KK(C_0(\cF^*),C^*(M;\cF))$. As in the case of \cite{Monthubert-Pierrot}, we will show that this element coincides with the analytic index.

\subsection{The ``adiabatic foliation''}

\bigskip Let $\lambda :M\times \R\to\R$ be the second projection and $J=\lambda C_c^\infty(M\times \R)$ the set of smooth compactly supported functions on $M\times \R$ which vanish on $M\times \{0\}$. Every vector field $X$ of $M$ extends to a vector field in $M \times \R$ tangent along $M$, which we will denote $X\otimes 1$. 

We define $\cT \cF$ to be the submodule of $C_c^\infty (M\times \R;TM\times \R)$ generated by $\lambda(\cF\otimes 1)$: it is the set of finite sums $\sum f_i(X_i\otimes 1)$ where $f_i\in J$ and $X_i\in \cF$.

\begin{proposition}
$\cT \cF$ is a foliation on $M\times \R$.
\end{proposition}

\begin{proof}
\begin{itemize}
\item Let $U$ be an open subset  of $M$ over which $\cF$ is generated by vector fields $X_1,\ldots ,X_k$. On $U\times \R$, $\cT \cF$ is generated by the vector fields $\lambda(X_1\otimes 1),\ldots ,\lambda(X_k\otimes 1)$. It follows that $\cT \cF$ is locally finitely generated.
\item If $f,g\in J$ and $X,Y\in \cF$, we find $$[f(X\otimes 1),g(Y\otimes 1)]=f(X\otimes 1)(g)(Y\otimes 1)-g(Y\otimes 1)(f)(X\otimes 1)+fg([X,Y]\otimes 1)\in \cT \cF.$$ It follows that $\cT \cF$ is integrable.\qedhere
\end{itemize}
\end{proof}

\begin{definition}
The foliation $(M\times \R,\cT \cF)$ is called the \emph{adiabatic foliation} associated with $\cF$.
\end{definition}

\begin{remark} Recall (\cite[Def. 1.2]{AndrSk}) that associated with a foliation $(M_0,\cF_0)$ are two families of vector spaces indexed by $M_0$: the space tangent to the leaf and the fiber of the module $\cF_0$.

The tangent subspace to the leaves at a point $(x,\beta)$ is $F_x\times \{0\}$ for $\beta\in \R^*$ and $\{(0,0)\}$ if $\beta=0$. On the other hand, the module $\cT\cF$ is isomorphic (via multiplication by $\lambda$) to the module $\cF\otimes 1$. In particular, these modules have the same fibers. It thus follows that $\cT\cF_{(x,\beta)}\simeq \cF_x$ for all $\beta\in \R$.

Also the total space of the cotangent ``bundle'' is $\cT\cF^*=\cF^*\times\R$.
\end{remark}

\subsection{The holonomy groupoid of the adiabatic foliation}

In order to describe a natural family of bi-submersions associated with $\cT\cF$, we will use the classical construction of \emph{deformation to the normal cone.} A complete account of this construction can be found \eg in \cite{Carrillo}. We just recall here a few facts about this construction:

\begin{itemize}
\item Let $U$ be a smooth manifold and $V$ a (locally closed) smooth submanifold of $U$. The \emph{deformation to the normal cone} of $U$ along $V$ is a smooth manifold $D(U,V)$ which as a set is $U\times \R^*\cup N\times \{0\}$ where $N$ is the (total space of the) normal bundle of the inclusion $V\subset U$ (\ie $N_x=T_xU/T_xV$ for $x\in V$). 

\item This construction is functorial (\cf \cite[3.4]{Carrillo}). Namely, if $(U,V)$ and $(U',V')$ are pairs of a manifold and a submanifold, then a smooth map $p : (U,V) \to (U',V')$ such that $p(V) \subset V'$ induces a (unique) smooth map $\tilde{p} : D(U,V) \to D(U',V')$ defined by $\tilde{p} = (p,\id)$ on $U \times \R^{*}$ and $\tilde{p}(x,n,0) = (p(x),d_{N}(p_{x}),0)$ for every $(n,0) \in N \times \{0\}$. Here $d_{N}p_{x}$ is by definition the map $(N)_{x} \to (N')_{p(x)}$ induced by $p$.

The map $\tilde p :D(U,V) \to D(U',V')$ is a submersion if and only if the map $p:U\to U'$ and its restriction $p_V:V\to V'$ are submersions.

\item Let us already notice that there is a smooth map $q:D(U,V)\to U\times \R\,(=D(U,U))$ which is the identity (and a diffeomorphism) on $U\times \R^*$ and such that $q(y,0)=(p(y),0)$ for $y\in N$ where $p : N\to V$ is the bundle projection.
\end{itemize}

\begin{proposition}\label{bisubmD}
Let $(U,t,s)$ be a bi-submersion for $\cF$ and $V\subset U$ be a closed identity bisection. \begin{enumerate}
\item Then $(D(U,V),t\circ q,s\circ q)$ is a bi-submersion for the adiabatic foliation $(M\times \R,\cT \cF)$.

\item  If $(U',t',s')$ is a bi-submersion adapted to $(U,t,s)$ and $V'\subset U'$ is any closed identity bisection of $U'$ such that $s'(V')\subset s(V)$, then $D(U',V')$ is adapted to $D(U,V)$.
\end{enumerate}

\end{proposition}
\begin{proof}
\begin{enumerate}
\item The maps $s\circ q$ and $t\circ q$ are the maps $\tilde{s}, \tilde{t}$ from $D(U,V)$ to $D(M,M)=M\times \R$ associated with the smooth submersions $s$ and $t$, whose restrictions to $V$ are \'etale. It follows that they are smooth submersions.

The assertion is local: we may restrict to a small open neighborhood of a given point $u\in U$. The restriction to the open set $U\setminus V$ is easy: we are dealing with the maps $s\times \id_{\R^*},t\times \id_{\R^*}:(U\setminus V)\times \R^*\to M\times \R^*$ which is easily seen to be a bi-submersion for the product foliation $\cF\otimes 1$. Now, the restrictions to the open set on $M\times \R^*\subset M\times \R$ of $\cT\cF$ and $\cF\otimes 1$ coincide.

Take now an open neighborhood of a point $v\in V$. We may therefore assume that $V$ is an open subset in $M$, $U$ is an open subset in $V\times \R^k$, and $s(v,\alpha)=v$. Restricting to an even smaller neighborhood of $v$, we may further assume that $t(U)\subset V$ and that $t$ has also a product decomposition. Therefore, $\ker dt$ is spanned by vector fields $(Y_1,\ldots,Y_k)$. Decompose each of these vector fields as $Y_i=(Z_i,Z'_i)$, where $Z_i$ is tangent along $V$ and $Z'_i$ is tangent along $\R^k$. The fact that $U$ is a bi-submersion means exactly that the $Z_i$ generate the foliation $\cF\otimes 1$ on $U$.

Now $D(U,V)$ identifies with $\{(v,\alpha,\beta)\in V\times \R^k\times \R;\ (v,\beta\alpha)\in U\}$; under this identification, $\tilde s(v,\alpha,\beta)=(v,\beta)$ and $\tilde t(v,\alpha,\beta)=(t(v,\beta\alpha),\beta)$. It follows that $d\tilde t_{(v,\alpha,\beta)}(Z,Z',0)=(dt_{(v,\beta\alpha)}(Z,\beta Z'),0)$ (for $Z$ tangent along $V$ and $Z'$ along $\R^k$), whence $\ker d\tilde t$ is spanned by $(\tilde Y_1,\ldots,\tilde Y_k)$ where $\tilde Y_i=(\tilde Z_i,\tilde Z'_i,0)$, the vector fields $\tilde Z_i$ and $\tilde Z'_i$ being defined by $\tilde Z_i(v,\alpha,\beta)=\beta Z_i(v,\alpha\beta)$ and $\tilde Z'_i(v,\alpha,\beta)=Z'_i(v,\alpha\beta)$. It follows that $\ker d\tilde s\oplus \ker d\tilde t$ is the set of vector fields $(Z,Z',0)$ where $Z'$ is any section of $\R^k$ (this is $\ker d\tilde s$) and $Z$ is in the module spanned by $\tilde Z_i$.

This proves that $\tilde s^{-1}(\cF)=\ker d\tilde s\oplus \ker d\tilde t$. Exchanging the roles of $s$ and $t$, we get the equality $\tilde t^{-1}(\cF)=\ker d\tilde s\oplus \ker d\tilde t$.

\item We may again consider two cases: the case where $V'$ is empty and the case where we deal with a small neighborhood of $v'\in V'$. It follows from \cite[Prop. 2.5]{AndrSkPsdo}, that there is a (local) morphism of bi-submersions mapping $V'$ to $V$. In both cases, we may assume that we have a morphism of bi-submersions $f:(U',t',s')\to (U,t,s)$ such that $f(V')\subset V$. By functoriality of the deformation to the normal cone, we obtain a smooth map $\tilde f:(\tilde U',\tilde t',\tilde s')\to (\tilde U,\tilde t,\tilde s)$ which is a morphism of bi-submersions. \qedhere
\end{enumerate}
\end{proof}

\begin{proposition} \label{propatlas}
Let $\cU=(U_i,t_i,s_i)_{i\in I}$ be an atlas for $(M,\cF)$ and let $V_i\subset U_i$ be identity bi-sections {\rm (\footnote{Some of the $V_i$'s may be empty})}. Assume that $\bigcup _{i\in I}s_i(V_i)=M$. \begin{enumerate}
\item Then $\widetilde\cU=(D(U_i,V_i),t_i\circ q_i,s_i\circ q_i)_{i\in I}$ is an atlas for $(M\times \R,\cT \cF)$. 
\item If moreover $\cU$ is the path holonomy atlas for $(M,\cF)$, then $\widetilde\cU$ is the path holonomy atlas for $(M\times \R,\widetilde\cF)$.
\end{enumerate}
\end{proposition}

\begin{proof}
\begin{enumerate}
\item Since $\tilde s(V_i\times \R)=V_i\times \R$, it follows that $\bigcup _{i\in I}\tilde s_i(D(U_i,V_i))=M\times \R$.

Let $(U,t,s)$ be a bi-submersion adapted to $\cU$ and $V\subset U$ be a closed identity bisection. It follows from prop. \ref{bisubmD}.b) that $(D(U,V),\tilde t,\tilde s)$ is adapted to $\widetilde\cU$. 

It follows that the inverse $(D(U,V),\tilde s,\tilde t)$ of $(D(U,V),\tilde t,\tilde s)$ is adapted to $\tilde \cU$ since $(U,s,t)$ is adapted to $\cU$.

If $(U',t',s')$ is another bi-submersion adapted to $\cU$ and $V'\subset U'$ is a closed identity bisection, one easily identifies the composition $(D(U,V),\tilde t,\tilde s)\circ (D(U',V'),\tilde t',\tilde s')$ with the bi-submersion $(D(U\circ U',V\circ V',(t\times \id)\circ q,(s\times \id)\circ q)$.

\item From the above arguments, it follows that if $\cU$ is generated by a subfamily $(U_i)_{i\in J}$ such that $\bigcup _{i\in J}s_i(V_i)=M$, then $\tilde \cU$ is generated by $(D(U_i,V_i))_{i\in J}$. Now, the path holonomy atlas is generated by a family $(U_i,s_i,t_i)$ of bi-submersions with identity bisections $V_i$ such that $s_i(U_i)=t_i(U_i)=s_i(V_i)$ and with connected fibers. Then $(D(U_i,V_i),\tilde t_i,\tilde s_i)$ satisfies the same properties. It generates the path holonomy atlas for $(M\times \R,\widetilde\cF)$. 
\qedhere
\end{enumerate}
\end{proof}

\begin{proposition}\label{tanfol}
Let $\cU$ be an atlas for $(M,\cF)$ and $\tilde \cU$ the corresponding atlas for $(M\times \R,\cT\cF)$. The groupoid of the atlas $\tilde \cU$ naturally identifies with $\bigcup_{x\in M}\cF_x\times \{0\}\cup \cG(\cU)\times \R^*$.

\begin{proof} 
Since the equivalence relation defining $\cG(\tilde \cU)$ respects the source and target maps, we find by composition  with the second projection a well defined map $\tau:\cG(\tilde \cU)\to \R$. Therefore $\cG(\tilde \cU)$ is the union  $\tau^{-1}(\R^*) \cup \tau^{-1}(\{0\})$. We will now identify $\tau^{-1}(\R^*)$ with $\cG(\cU)\times \R^*$ and $\tau^{-1}(\{0\})$ with $\bigcup_{x\in M}\cF_x$. 

\begin{enumerate}
\item Let $(W,t,s)$ be a bi-submersion adapted to $\tilde \cU$. For $\beta \in\R$, put $W_\beta=s^{-1}(M\times \{\beta\})$. For $\beta \ne 0$, by restriction of $t,s$ to $W_\beta$ we get a bi-submersion $(W_\beta,t_\beta,s_\beta)$ adapted to $\cU$. Also if $(W',t',s')$ is adapted to $(W,t,s)$ the restriction $(W'_\beta,t'_\beta,s'_\beta)$ is adapted to $(W_\beta,t_\beta,s_\beta)$.
We have constructed a map $P_{\R^*}:\tau^{-1}(\R^*)\to \cG(\cU)\times \R^*$.

Let $(U,t,s)$ be a bi-submersion adapted to the atlas $\cU$. Putting $V=\emptyset$, we find a bi-submersion $(U\times \R^*,t\times \id_{\R^*},s\times \id_{\R^*})$ adapted to $\tilde \cU$.  Also if $(U',t',s')$ is adapted to $(U,t,s)$ then $(U'\times \R^*,t'\times \id_{\R^*},s'\times \id_{\R^*})$ is adapted to $(U\times \R^*,t\times \id_{\R^*},s\times \id_{\R^*})$. This way we construct a map $\cG(\cU)\times \R^*\to \tau^{-1}(\R^*)$, which is easily seen to be inverse to $P_{\R^*}$.

\item 
Let also  $V\subset U$ be an identity bisection. Assuming that $s$ is injective on $V$, we identify $V$ with its image in $M$ which is an open subset of $M$. Consider the map $dt-ds$ which to a vector field $\xi\in C_c^\infty(V;TU)$ associates the vector field $dt(\xi)-ds(\xi)\in C_c^\infty(V;TM)$. By definition of a bi-submersion, its range lies in $\cF$. Note that since $ds $ and $dt$ coincide for vectors along $V$, $(dt-ds)(\xi)$ only depends on the normal part of $\xi$, \ie its image in $C_c^\infty (V;NV)=C_c^\infty (V;TU/TV)$; we get in this way a map $\Phi:C_c^\infty(V;NV)\to \cF$ which is $C^\infty(M)$ linear - \ie a module map. At each point of $V$, we get a map between the fibers $q_x^V:N_xV\to \cF_x$. Again, if $(U',t',s')$ is another bi-submersion carrying the identity at $x$ and $V'$ is an identity bisection through $x$, then we have a morphism $j_x:N_xV'\to N_xV$ and it is easily seen that $q_x^{V'}=q_x^V\circ j_x$, so that we constructed a map $\cG(\tilde U)_{(x,0)}\to \cF_x$. 

We have constructed a map $P_0:\tau^{-1}(\{0\})\to \bigcup_{x\in M}\cF_x$.

The image of the map $\Phi$ is (again by definition of bi-submersions) the space  $C_c^\infty(V). \cF$ of elements of $\cF$ with support in $V$. It follows that $P_0$ is onto.

Now, if $U$ is minimal at $x$, then $q_x$ is injective, and it follows that $P_0$ is injective.\qedhere
\end{enumerate}
\end{proof}
\end{proposition}

\subsection{The short exact sequence}

\emph{As explained above (page \pageref{section1.6}), from now on we fix an atlas $\cU$ for $(M,\cF)$ and the corresponding atlas $\tilde \cU$ for $(M\times \R,\cT\cF)$. All bi-submersions considered here are assumed to be adapted to this atlas. Also what we call $C^*(M,\cF)$ and $C^*(M\times \R,\cT\cF)$ are in fact the (full) $C^*$-algebras associated with these atlases.}

Here we construct a short exact sequence of $C^*$-algebras 
\begin{eqnarray}\label{extn1}
0 \to C_{0}(\R^*) \otimes C^{*}(M,\cF) \ {\buildrel j\over\longrightarrow}\  C^{*}(M \times \R,\cT \cF) \ {\buildrel \pi\over\longrightarrow}\ C_{0}(\cF^{*}) \to 0
\end{eqnarray}

We first identify $C_{0}(\R^*) \otimes C^{*}(M,\cF)$ with an ideal in $C^{*}(M \times \R,\cT \cF)$, then construct the homomorphism $\pi$, show that it is onto, and finally identify the kernel of $\pi$ with the image of $C_{0}(\R^*) \otimes C^{*}(M,\cF)$.

\subsubsection{Construction of $j$}
The $C^*$-algebra of the restriction of $\cT \cF$ to $M\times \R^*$ is an ideal $J$ in $C^*(M\times \R;\cT\cF)$. Now, the restriction of $\cT \cF$ to $M\times \R^*$ coincides with $\cF\otimes 1$. Evaluation at each point $\beta$ of $\R^*$, gives a map: $\ev_\beta:J\to C^*(M,\cF)$. By density of $C_c^\infty (U)\otimes C_c^\infty(\R^*)$ in $C_c^\infty (U\times \R^*)$ with respect to the $L^1$-estimate (\cite[\S4.4]{AndrSk}), it follows that for every $x\in J$, the map $\beta\mapsto \ev_\beta(x)$ is continuous. In this way, we constructed a $*$-homomorphism $J\to C_{0}(\R^*) \otimes C^{*}(M,\cF)$. Using again functions in $C_c^\infty (U)\otimes C_c^\infty(\R^*)$, it follows that this map is onto. 

To show that it is injective, we have to show that every irreducible representation $\theta$ of $J$ factors through $C_{0}(\R^*) \otimes C^{*}(M,\cF)$. But representations of $J=C^*(M\times \R^*;\cF\otimes 1)$ were described in \cite[\S5]{AndrSk} and in particular they give rise to a measure on $M\times \R^*$.
Denote by $\overline \theta$ the extension of $\theta $ to the multipliers. Since $C_0(\R^*)$ lives in the center of the multipliers of $J$, $\overline\theta(C_0(\R^*))$ lies in the center of the bi-commutant of $\theta$, and is therefore scalar valued. In other words, $\overline \theta$ is a character of $C_0(\R^*)$. It follows that there exists $\beta\in \R$ such that this measure is carried by $M\times \{\beta\}$. The representation $\theta$ is really a representation of the groupoid $\cG(\cU)\times \R^*$, and since the corresponding measure is carried by $M \times\{\beta\}$ it is in fact a representation of the groupoid $\cG(\cU)\times \{\beta\}$. It follows, that $\theta$ is of the form $\theta'\circ\ev_\beta$.

\subsubsection{Construction of $\pi$.}

Let $(U,t,s)$ be a bi-submersion for $(M,\cF)$ and $V$ an identity bisection. Put $\widetilde{U} = D(U,V)$. We define a map $\varpi_{(U,V)} : C_{c}^{\infty}(\widetilde{U};\Omega^{1/2}\widetilde{U}) \to C_{0}(\cF^{*})$ as follows: Given $f \in C_{c}^{\infty}(\widetilde{U};\Omega^{1/2}\widetilde{U})$,
\begin{itemize}
 \item first restrict it to $f_{0} \in C^{\infty}_{c}(NV \times \{0\}; \Omega^{1}NV)$;
 \item then apply the Fourier transform to obtain $\widehat{f}_{0} \in C_{0}(N^{*}V)$;
 \item since $\cF^*_V=\{(x,\xi);\ x\in s(V),\ \xi\in \cF^*_x\}$ identifies with a closed subspace of $N^{*}V$, consider the restriction to this set and extend it by $0$ outside $\cF_V^*$ to get an element $\varpi_{(U,V)}(f)=\widehat{f}_{0}\mid_{\cF^{*}} \in C_{0}(\cF^{*})$.
\end{itemize}

We next show that $\pi$ is a well defined and surjective homomorphism.

To show that $\pi $ is a well defined homomorphism, we just need to show that for every $x\in M$ and $\xi\in \cF^*_x$ there is a well defined character $\hat \chi_{(x,\xi)}$ of $C^*(M\times \R;\cT\cF)$ such that the image of the class of $f\in C_{c}^{\infty}(\widetilde{U};\Omega^{1/2}\widetilde{U})$ is $\varpi_{(U,V)}(f)(x,\xi)$. Now, $\xi$ defines a one dimensional representation of the groupoid $\cG(\widetilde{\cU})$ in the sense of \cite[\S5]{AndrSk}: the corresponding measure is the Dirac measure $\delta_{(x,0)}$ on $ M\times \R$, the Hilbert space is just $\C$, and $\chi_\xi(x,X)=e^{-i\langle X|\xi\rangle}$ for $X\in \cF_x$ (the rest of the groupoid being of measure $0$, the value of $\chi_\xi$ on an element which is not of the form $(x,X)$ doesn't matter). It is now an elementary calculation to see that the image of $f$ under the character $\hat \chi_{(x,\xi)}$ of $C^*(M\times\R;\cT\cF)$ corresponding to $(\delta_{(x,0)},\C,\chi_\xi)$ is $\varpi_{(U,V)}(f)(x,\xi)$.

To show that $\pi$ is onto, first note that the map $f\mapsto f_0$ is surjective from $C_{c}^{\infty}(\widetilde{U};\Omega^{1/2}\widetilde{U})$ to $C^{\infty}_{c}(NV \times \{0\}; \Omega^{1}NV)$, that the Fourier transform has then dense range in $C_0(N^*V)$ hence the closure of the image by $\varpi_{(U,V)}$ of $C_{c}^{\infty}(\widetilde{U};\Omega^{1/2}\widetilde{U})$ is the set of functions on $\cF^*$ which vanish outside the open set $\cF_V$. Since the $s(V_i)$ form an open cover of $M$, these sets form an open cover of $\cF^*$. It follows that $\pi$ is surjective.

\subsubsection{Exactness}

We now come to the main result of this section:

\begin{theorem}
The sequence {\rm(\ref{extn1})}, namely\begin{eqnarray*}
0 \to C_{0}(\R^*) \otimes C^{*}(M,\cF) \ {\buildrel j\over\longrightarrow}\  C^{*}(M \times \R,\cT \cF) \ {\buildrel \pi\over\longrightarrow}\ C_{0}(\cF^{*}) \to 0
\end{eqnarray*}
is exact
\begin{proof} 
We already showed that $j$ is injective and $\pi $  is surjective. One sees also easily that $\pi\circ j=0$.

Put $A = C^{*}(M \times \R, \cT\cF)$ and $J=j(C_{0}(\R^*) \otimes C^{*}(M,\cF))$.

Let $\tilde U=D(U,V)$ be a bi-submersion and $f\in C_{c}^{\infty}(\widetilde{U};\Omega^{1/2}\widetilde{U})$; if $f$ vanishes in a neighborhood of $NV\times \{0\}$, its image in $A$ lies in $J$; the same holds if $f$ just vanishes on $NV\times \{0\}$ thanks to the $L^1$ estimate (\cite[\S4.4]{AndrSk}). Indeed, $f$ can then be approximated uniformly with fixed support by a sequence of elements which vanish near $NV$.

It suffices to show that every irreducible representation of $A$ that vanishes on $J$ vanishes on $\ker \pi$. So we'll just show that every such irreducible representation $\theta$ is actually a point of $\cF^{*}$. Extending $\theta $ to the multipliers, we find a representation $\overline \theta$ of $C_0(M\times \R)$.

Take $f \in C_c^{\infty}(M\times \R)$ and $g \in C^{\infty}_{c}(\widetilde{U})$, and put $h = (f\circ \widetilde{t})g - g(f \circ \widetilde{s}) \in C^{\infty}_{c}(\widetilde{U})$ which vanishes on $NV \times \{0\}$. Therefore $\overline \theta(f)\theta(g) - \theta(g)\overline \theta(f)=\theta (h)=0$. It follows that $\overline \theta(C_0(M\times \R))$ lies in the commutant $\C1$ of $\theta$ and thus $\overline \theta$ is a character, \ie a point $(x,\beta) \in M\times \R$. Now if $f$ vanishes in a neighborhood of $M\times \{0\}$, then $fA\subset J$. It follows that $\beta =0$.

 By \cite[\S 5]{AndrSk} $\theta$ is the integrated form of a representation $(\mu,\cH,\chi)$ of the groupoid $\cG(\tilde \cU)$. We just showed that the measure $\mu$ is a Dirac measure $\delta_{(x,0)}$, and it follows that $\chi $ is just a representation of $\cF_x$ on the Hilbert space $\cH$, whence a direct integral of characters $\chi_\xi$. Therefore, $\theta$ is itself a direct integral of characters $\hat\chi_{(x,\xi)}$. Since it is irreducible it coincides with a character $\hat\chi_{(x,\xi)}$.
 \end{proof}
\end{theorem}

\begin{remark}\label{evt}
Since $C^*(M\times\R;\cT\cF)$ is a $C_0(\R)$ algebra it restricts to any locally closed subset of $\R$. If $Y=T_1\setminus T_2$ where $T_2\subset T_1$ are open sets of $\R$, one puts $C^*(M\times\R;\cT\cF)_Y=C_0(T_1)C^*(M\times\R;\cT\cF)/C_0(T_2)C^*(M\times\R;\cT\cF)$.

Restricting extension (\ref{extn1}) to $[0,1]$, we get an exact sequence:\begin{eqnarray}\label{extn2}
0 \to C_{0}((0,1]) \otimes C^{*}(M,\cF) \to C^{*}(M \times \R,\cT \cF)_{[0,1]} \ {\buildrel \ev_0\over\longrightarrow}\ C_{0}(\cF^{*}) \to 0
\end{eqnarray}
\end{remark}

\section{The analytic index via the tangent groupoid}

The tangent groupoid exact sequence (\ref{extn2}) gives rise to an element in $ KK(C_0(\cF^*),C^*(M,\cF))$ and we will show that this element coincides with the analytic index element.

Indeed, since $C_0(\cF^*)$ is abelian, the exact sequence (\ref{extn2}) is semi-split. Moreover the kernel of the homomorphism  $\ev_0$  is the contractible $C^*$-algebra $C_{0}((0,1]) \otimes C^{*}(M,\cF)$, so that the element $[\ev_0]\in KK(C^{*}(M \times \R,\cT \cF)_{[0,1]},C_0(\cF^*))$ is invertible. 

Evaluation at $1$ is a morphism $\ev_1:C^{*}(M \times \R,\cT \cF)_{[0,1]}\to C^*(M,\cF)$.

The main result in this section is:

\begin{theorem}
We have the equality $\ind_a=[\ev_0]^{-1}\otimes [\ev_1]\in KK(C_0(\cF^*),C^{*}(M,\cF))$.
\end{theorem}
\begin{proof}
The restriction to $[0,1]$ of the exact sequence of pseudodifferential operators on $\cT\cF$ is written as follows: $$0\to C^*(M\times \R;\cT\cF)_{[0,1]}\to \Psi(M\times \R;\cT\cF)_{[0,1]}\to B_{[0,1]}\to 0.$$ Here $B_{[0,1]}$ is a quotient of $C_0(S^*\cF\times [0,1])$.

Extending functions on $S\cF^*$ to $S\cF^*\times [0,1]$ (by taking them independent on the variable in $[0,1]$) we get a morphism $C(S\cF^*)\to B_{[0,1]}$. Also considering multiplication by functions on $M$ as pseudodifferential elements, we get a diagram 

\begin{align}\label{diag6}
\xymatrix{
&& C(M)\ar[d]_{\tilde m}\ar[r]^{ p} & C(S\cF^*)\ar[d]^{\tilde q} \\
0 \ar[r] & C^{*}(M \times\R,\cT\cF)_{[0,1]}  \ar[r] & \Psi(M\times \R;\cT\cF)_{[0,1]}  \ar[r]^{\quad \qquad\widetilde{\sigma}} &  B_{[0,1]}   \ar[r] & 0 
}
\end{align}
from which we get a morphism $\tilde\varphi:C_0(\cF^*)\simeq\cC_p\to \cC_{\tilde \sigma}$ and a $KK$-element $$\widetilde{\ind_a}=[\tilde \varphi]\otimes_{\cC_{\tilde \sigma}}[\tilde e]^{-1}\in KK(C_0(\cF^*),C^{*}(M \times\R,\cT\cF)_{[0,1]})$$ where $\tilde e:C^{*}(M \times\R,\cT\cF)_{[0,1]}\to \cC_{\tilde \sigma}$ is the excision morphism.

The theorem is an immediate consequence of the following two facts:\begin{description}
\item[Claim 1.] $(\ev_1)_*(\widetilde{\ind_a})=\ind_a$
\item[Claim 2.] $\widetilde{\ind_a}=[\ev_0]^{-1}$.
\end{description}

\emph{Proof of Claim 1.}
Evaluation at $1$ gives the following diagram: $$\xymatrix{
0 \ar[r] & C^{*}(M \times\R,\cT\cF)_{[0,1]} \ar[d]^{\ev_{1}} \ar[r] & \Psi(M\times \R;\cT\cF)_{[0,1]} \ar[d]^{\ev_{1}^{\Psi}} \ar[r]^{\quad \qquad\widetilde{\sigma}} &  B_{[0,1]} \ar[d]^{\ev_{1}^{B}}   \ar[r] & 0 \\ 
0 \ar[r] & C^{*}(M,\cF) \ar[r] & \Psi(M,\cF) \ar[r]^{\qquad\sigma} & B_1 \ar[r] & 0
}$$
from which we get a commuting diagram:
\begin{align}\label{diag15}
\xymatrix{
 C^{*}(M \times\R,\cT\cF)_{[0,1]} \ar[d]_{\ev_{1}} \ar[r]^{\quad \qquad \tilde e} &\cC_{\tilde \sigma} \ar[d]^{\ev_{1}^{\cC}}   \\ 
C^{*}(M,\cF) \ar[r]^{e} & \cC_{\sigma} 
}
\end{align}
Moreover, $\ev_1^B\circ \tilde q=q$ and $\ev_1^\Psi\circ \tilde m=m$, whence $\ev_1^\cC\circ \tilde \varphi=\varphi$.

We thus have
\begin{eqnarray*}
(\ev_1)_*(\widetilde{\ind_a})&=&[\tilde \varphi]\otimes_{\cC_{\tilde \sigma}}[\tilde e]^{-1}\otimes [\ev_1]=[\tilde \varphi]\otimes_{\cC_{\tilde \sigma}}[\ev_1^\cC]\otimes_{\cC_\sigma}[e]^{-1}\\ &=&[\ev_1^\cC\circ \tilde \varphi]\otimes_{\cC_\sigma}[e]^{-1}=\ind_a
\end{eqnarray*}

\emph{Proof of Claim 2.} Evaluation at $0$ gives the following diagram: 
\begin{align}\label{diag8}
\xymatrix{
0 \ar[r] & C^{*}(M \times\R,\cT\cF)_{[0,1]} \ar[d]^{\ev_0} \ar[r] & \Psi(M\times \R;\cT\cF)_{[0,1]} \ar[d]^{\ev_0^{\Psi}} \ar[r]^{\quad \qquad\widetilde{\sigma}} &  B_{[0,1]} \ar[d]^{\ev_0^{B}}   \ar[r] & 0 \\ 
0 \ar[r] & C_{0}(\cF^{*}) \ar[r] & \Psi(M\times \R;\cT\cF)_0\ar[r]^{\ \qquad \sigma_{0}} & B_0 \ar[r] & 0}\end{align}
Let $x\in M$. Let $(U,t,s)$ be a bi-submersion for $(M,\cF)$ and $V$ an identity bisection such that $x\in s(V)$. Assume that $U$ is minimal at $x$. Then $V\times \R\subset D(U,V)$ is an identity bisection (for the foliation $(M\times \R,\cT\cF)$). 
Let $P\in \cP_c^0(D(U,V),V\times \R;\Omega^{1/2})$ be a pseudodifferential distribution with compact support (\cf \cite[\S 1.2.2]{AndrSkPsdo}). From the definition of the pseudodifferential family, it follows that for $\xi\in�\cF_x$,  we have $\hat \chi_{(x,\xi)}(P)=a(x,\xi,0)$ where $a$ is a symbol of $P$. It follows that the algebra $\Psi(M\times \R;\cT\cF)_0$ is the closure of order zero symbols, \ie the algebra $C(\overline {\cF^*})$ where $\overline {\cF^*}$ denotes the closure of $\cF^*$ by spheres at infinity (which is homeomorphic to the ``bundle'' of closed unit balls).

The bottom line in diagram (\ref{diag8}) is $$0\to C_0(\cF^*)\to C(\overline {\cF^*}){\buildrel p_0\over {\longrightarrow}} C(S^*\cF)\to 0.$$ 
Moreover $\ev_0^B\circ \tilde q$ is the identity of $C(S\cF^*)$. Therefore $(\ev_0)_*(\widetilde{\ind_a})=[e_0]^{-1}\otimes [\psi]$ where $e_0:C_0(\cF^*)\to \cC_{p_0}$ is the excision map and $\psi:\cC_{p}\to \cC_{p_0}$ is the morphism corresponding to the commuting diagram
$$
\xymatrix{
C(M) \ar[d] \ar[r]^{p} & C(S^*\cF) \ar@{=}[d] \\ 
C(\overline {\cF^*}) \ar[r]^{p_0} & C(S^*\cF)
}
$$
But one easily identifies $\cC_{p_0}$ with $C_0(\cF^*)$ in such a way that both $\psi$ and $e_0$ are homotopic to the identity. It follows that $(\ev_0)_*(\widetilde{\ind_a})=1_{C_0(\cF^*)}$.
\end{proof}

\end{document}